\documentclass[12pt]{article}

\usepackage{amsmath}
\usepackage{amssymb, amsthm}
\usepackage{url}
\usepackage[all]{xy}

\newtheorem{lemma}{Lemma}[section]
\newtheorem{prop}[lemma]{Proposition}
\newtheorem{thm}[lemma]{Theorem}
\newtheorem{defn}[lemma]{Definition}

\newtheorem{remark}[lemma]{Remark}

\DeclareMathOperator \sy {sym}
\DeclareMathOperator \su {sum}

\newcommand{\AQ}{A}

\newcommand{\mr}[1]{\stackrel{#1}{\longrightarrow}}
\newcommand{\mmr}[1]{\buildrel {#1} \over \hookrightarrow}

\date{}

\title{Column symmetric polynomials}

\author{Eduardo Dubuc  and Anders Kock}

\begin{document}

\maketitle 

\section*{Introduction}
Let $\AQ$ be a commutative ring. It is classical how symmetric polynomials in $ A[x_{1}, \ldots 
,x_{n}]$ are uniquely expressible as polynomials in the $n$ elementary 
symmetric polynomials, cf.\ e.g.\ \cite{[vdWI]} \S 29. For instance for 
$n=2$, the two elementary polynomials are $\sigma _{1}:=x_{1}+x_{2}$ and 
$\sigma_{2}:=x_{1}x_{2}$; and the symmetric polynomial $x_{1}^{2} + x_{2}^{2}$ 
may be expressed as $\sigma_{1}^{2}- 2\sigma_{2}$:
$$x_{1}^{2} + x_{2}^{2}= (x_{1}+x_{2})^{2}- 2 x_{1}x_{2}  
.$$
Modulo the ideal $I$ generated by $x_{1}^{2}$ and $x_{2}^{2}$, we 
therefore also have
$$x_{1}x_{2}= \frac{1}{2}(x_{1}+x_{2})^{2},$$
provided $\frac{1}{2}$ exists in the base ring $A$. 

In fact, we have more generally that if $\AQ$ contains the ring 
${\mathbb Q}$ of rationals, then, modulo $I$,  any symmetric 
polynomial in $\AQ[x_{1}, \ldots 
,x_{n}]$ may be uniquely expressed as a polynomium in the 
single symmetric polynomium $x_{1}+ \ldots + x_{n}$, where $I$ is the 
ideal generated by the $x_{i}^{2}$s. This is a well known and 
important fact, called ``the symmetric functions 
property''  in  \cite{[SDG]} Exercise I.3.3 (quoting Dubuc and Joyal). 

It is a result in this direction we intend to generalize from 
dimension $1$ to dimension $m$. We are 
considering the polynomial ring in $m\times n$ variables $x_{i,j}$; the kind of 
symmetry we consider is not with respect to all the 
$m n$  
variables; we consider these variables organized in an $m\times n$ 
matrix,  
and we only consider invariance under the $n!$ permutations of the $n$ 
columns. The result refers to what we can assert, modulo 
the ideal $I$ generated by the degree 2 monomials 
$\{x_{ij} x_{i'j}\}_{j = 1, \ldots , n,\; i = 1, \ldots , m,\; i' = 1, \ldots , m}$.

 The result 
 asserts that any polynomial, invariant under the $n!$ 
 permutations of the columns can, modulo $I$, be expressed uniquely as a 
polymonial in the $m$ ``row-sums'', 
$\{s_i = x_{i,1} + x_{i,2} + \ldots +x_{i,n} \}_{i = 1, \ldots , m}$. 
The classical ``symmetric 
functions property'' is the special case where $m=1$.

An application of this Theorem concerns formal exactness of closed 
differential 1-forms is sketched in  Section \ref{formsx} below.

\medskip

Throughout $\AQ$ will be a commutative ring. It is assumed to contain 
${\mathbb Q}$.  All the 
$\AQ$-modules which we consider are free. Therefore, we  use 
terminology from linear algebra, as if $\AQ$ were a field.

\section{Polynomials in a matrix of variables} \label{polymatrix}

\subsection{ The free commutative monoid}\label{FCMx}

The free commutative monoid $M(X)$ on a set $X$ is in a natural way a
graded monoid. We call its elements {\em monomials} in $X$, we call $X$ the
set of {\em variables}; we write the monoid structure
multiplicatively. We shall give an explicit presentation of $M(X)$.\footnote{An equivalent description is that $M(X)$ is the 
set of finite {\em multi-subsets} of $X$.} 

Let $k$ be a positive integer; we 
let $[k]$ denote the set $[k]=\{1,2, \ldots ,k\}$.  
Then a  monomial $\omega$ of degree $k$ may be explicitly 
presented by a 
function $f: [k] \to X$; we write the monomium thus presented   
\mbox{$\omega_{f}:= x_{f(1)} x_{f(2)}  \ldots  x_{f(k)}$.} Since the
variables commute, it follows that
two functions $f$ and \mbox{$f':[k]\to  X$} present the same monomium
iff 
they differ by a permutation $ \varepsilon :[k]\to [k]$ of $[k]$,
i.e. $f'= f\circ \varepsilon$. 

Later on in the proof of Proposition \ref{injx}, 
we shall need a  finer notation:
We denote by $\lVert f\rVert $ the set of  all functions $f\circ \varepsilon$ 
for $\varepsilon \in S_{k}$ (where $S_{k}$ is the group of 
permutations of $[k]$). 
Thus  $\lVert f\rVert $ is the  
orbit of $f$ under the right action (by precomposition) of $S_{k}$.
The monomials are actually indexed by these orbits, we have a well defined monomium  $\omega _{\lVert f\rVert}$, and  
\mbox{$\omega_{\lVert f \rVert} =\omega_{\lVert f' \rVert} \iff 
  \lVert f \rVert = \lVert f' \rVert$}.

\subsection{The polynomial ring in a matrix of 
variables}  \label{admissibles}

If $A$ is any commutative 
ring, the polynomial ring $A[X]$ with coefficients in $A$ in a set
$X$ of indeterminates is
the free commutative $A$-algebra  on the set $X$. It may be
constructed by a two-stage process: first, construct the free
commutative monoid $M(X)$ on $X$, 
and then construct the free $A$-module %\rmv{$A[M(X)]$}%
on the set $M(X)$.  
 It inherits its multiplication from that of $M(X)$. 
It is a {\em graded} $A$-algebra, with the degree-$k$ part being the
linear submodule with basis  the monomials of degree $k$.

We shall be interested in some further structure which the algebra 
$A[X]$ has, in the case where the set $X$ is given as a
product set 
$[m]\times [n]$.   We think of this $X$ as the set of $m\times n$
matrices ($m$ rows, $n$ columns) with entries \mbox{$x_{i,j}$ ($i\in
[m], j\in [n]$),} and write \mbox{$A[M^{m\times n}]:= A[[m]\times [n]] = 
A[x_{1,1}, \ldots ,x_{m,n}].$}

A function $[k]\mr{} [m]\times [n]$ is given by a pair 
$(f,g) $, where $f:[k]\to [m]$ and $g: [k]\to  [n]$. The 
monomium presented by such function we denote 
$\omega _{(f,g)}$, or just 
$\omega_{f,g}$. Thus
\begin{equation} \label{omegafg}
\omega_{f ,g}= \prod _{l\in [k]} x_{f(l),g(l)} \;=\;
x_{f(1)g(1)} x_{f(2)g(2)} \; \ldots \; x_{f(k)g(k)}
\end{equation}
Clearly, when $g$ is monic, then so is any other $g'$, for any other 
presentation $( f',g' )$ of the same monomium. 
Therefore, the following notion is well defined. 
\begin{defn}The monomial $\omega_{f,g}$ is {\em 
admissible} if $g: [k]\to  [n]$ is monic.
A polynomium $\in \AQ[M^{m \times n}]$ is called {\em
admissible} if it is a linear 
combination of admissible monomials.
\end{defn}

So a monomium in the $x_{i,j}$'s is admissible if it does 
not contain two factors from any of the columns, like 
$x_{i,j}\cdot  x_{i',j}$. In particular, it 
does not contain any squared factor $x_{i,j}^{2}$. 
Clearly, admissible polynomials are of degree $\leq n$.

If $\omega$ is not admissible, it is called inadmissible. If $\omega$
is inadmissible, 
then so is $\omega \cdot \theta$ for any monomium $\theta$. It 
follows that the linear subspace of $\AQ [M^{m \times n}]$ 
generated by the inadmissible monomials is an ideal 
\mbox{$I\subseteq \AQ [M^{m \times n}]$.}
The quotient algebra $\AQ [M^{m \times n}]/I$ may be 
identified with the linear subspace (not a subalgebra)  
$\AQ _a[M^{m \times n}] \subseteq \AQ [M^{m \times n}]$
generated by the admissible monomials, with the projection morphism 
\mbox{$\AQ[M^{m \times n}] \mr{} \AQ _a[M^{m \times
n}]$}  being the map which discards all terms containing an
inadmissible factor. The   algebra structure of $\AQ _a[M^{m \times
n}]$ is thus given by the
multiplication table $\{x_{i,j}\cdot  x_{i',j} = 0 \}_{i \in [m],\; i' \in
[m],\; j \in [n]}$, and no other 
relations.\footnote{$\AQ _{a}[M^{m\times n}$] is an example of 
what sometimes is called a 
{\em Weil-algebra} over $\AQ$; in particular, it is 
finite-dimensional as an $ \AQ$-module. Likewise, the algebra $\AQ 
_{\leq n}[y_{1}, \ldots ,y_{m}]$ to be considered below, is a 
Weil-algebra.}
 The algebra $\AQ _a[M^{m \times n}]$ inherits a grading from
that 
of $\AQ [M^{m \times n}]$. 
Note that in $\AQ _a[M^{m \times n}]$
all non-zero elements are of degree $\leq n$.

\vspace{1ex}

Among the polynomials in $\AQ [M^{m \times n}]$ we have the $m$ 
``row-sums'' \mbox{$s_{i}$ \mbox{, $i=1, \ldots , m$}} 
(the sum of the entries in the $i$th row); they are all admissible:
\begin{equation} \label{row-sum}
s_{i}:= \sum _{j\in [n]} x_{i,j} \;=\; x_{i,1} + x_{i,2}, \,\ldots\, x_{i,n}.
\end{equation}
Consider any map $f: [k]\to [m]$. 
By the distributive law, 
i.e.\  by multiplying out the product,  we have the second 
equality sign in
$$\prod_{l\in [k]}s_{f(l)} = \prod_{l\in [k]}\sum _{j\in 
[n]}x_{f(l)j}   =\sum _{[k]\mr{g} [n]}\; \prod_{l\in [k]} x_{f(l)g(l)},$$
where $g$ ranges over the set of all maps $[k]\to [n]$.
 The admissible terms here are those where $g$ is injective, so 
 modulo $I$, equivalently,  
discarding inadmissible terms,   
\begin{equation}\label{twox}
\prod_{l\in [k]}s_{f(l)} \;=\; \sum _{[k] \mmr{g} [n]} \,
\prod_{l\in [k]} x_{f(l),g(l)} \hspace{3ex}  
\text{in the algebra $\AQ_a [M^{m \times n}]$.} 
\end{equation}
where $g$ now ranges over the set of monic maps $[k]\to [n]$.

\subsection{Column symmetric polynomials}

Let $\sigma$ be a permutation $\sigma : [n]\to [n]$, i.e.\  $\sigma
\in S_{n}$. One may permute the $n$ 
columns of the matrix $X$ of variables $x_{i,j}$ by $\sigma$. More
explicitly, 
$\sigma$ permutes 
 the monomials by the recipe: 
 \begin{equation}\label{actx}
   \sigma \cdot 
   \omega_{f,g}:= 
   \omega_{f ,  \sigma \circ  g} \, .
\end{equation}
This is well defined with respect to different
presentations of the same monomial. 
Thus, the set of monomials carry a left action by $S_{n}$. 
If $g: [k] \to [n]$ is injective, then so is $\sigma \circ  g$, for any 
permutation $\sigma : [n] \to [n]$, hence the subset of admissible
monomials is stable under the action. The action clearly extends to
an action on the polynomial algebras $\AQ [M^{m \times n}]$
and  $\AQ _a[M^{m \times n}]$. Note that the subspace
inclusion as well as the quotient morphism preserve the action.

The polynomials which are invariant under the action 
of $S_{n}$, we call {\em column symmetric}. These elements form subalgebras of 
$\AQ [M^{m \times n}]$ and of $\AQ _a[M^{m \times n}]$,
 and they  deserve the notation  
$\sy (\AQ [M^{m \times n}])$ and $\sy (\AQ _a[M^{m \times
n}])$, respectively.

In the sequel we study the structure of the elements of the algebra  
$\sy (\AQ _a[M^{m \times n}]) \subseteq 
                                \AQ _a[M^{m \times n}]$. This is where 
								  we need that the ring $\AQ$ contains 
								  ${\mathbb Q}$.

								  If a finite group $S$ acts on an algebra 
								  $C$ over a commutative ring
$A$, the elements in $C$ invariant 
under the action of $S$ form a subalgebra $\sy _{S} (C)$ of 
{\em $S$-symmetric} or {\em $S$-invariant} elements. If $A$ contains
the field of rational numbers ${\mathbb Q}$ as a subring, we further
have that the subalgebra $\sy _{S}(C)\subseteq C$,  seen just as a
linear subspace, is a retract, with 
retraction the symmetrization operator $\sy$ given, for $a\in C$, by
\begin{equation}  \label{opx}
\sy (a):= p^{-1}\cdot \sum _{\sigma \in S}\sigma \cdot a,\end{equation}
where $p$ is the cardinality of $S$. And we have
$$ a \; is \; invariant  \iff  \; a = \sy (a).$$

\begin{prop}
Any two admissible monomials $\omega _{f,g}\,$,  
$\omega_{f,g'}$ with the same \mbox{$f: [k]\to [m]$} are in the same orbit
of the action by $S_{n}$. It follows that
$\sy (\omega_{f,g}) = \sy (\omega_{f,g'})$, see (\ref{opx}). 
\end{prop}
\begin{proof} Recall that if $g$ and $g':$ $[k] \to [n]$ are monic,
then  we may find a permutation $[n] \mr{\tau} [n]$ 
with $\tau \circ g = g'$. There are in fact $(n-k)!$ such 
permutations. With such $\tau$, we have $\tau \cdot \omega_{f, 
g}= \omega _{f, g'}$. It follows that 
$\sy (\omega_{f,g})$ and $\sy (\omega_{f,g'})$ have the same terms
but in different order.
\end{proof}

The row-sum polynomials $s_{i}\,$, see (\ref{row-sum}), are
clearly column-symmetric, and the product 
$\prod_{l\in [k]}s_{f(l)}$, as a $k$-fold product of homogeneous
degree 1 polynomials, is a homogeneous degree $k$ polynomial, and 
likewise column symmetric. 
 
 \begin{prop} \label{main1x} For  any admissible monomium 
 $\omega_{f,g}$ of degree $k$, we have (discarding inadmissible terms)
  $$
 \sy (\omega_{f,g}) =  
 \frac{(n-k)!}{n!} \prod_{l\in [k]}s_{f(l)}.
 $$
 \end{prop}
\begin{proof}
Given a monic map $[k] \mmr{g} [n]$, let $\sigma$ be any permutation 
$[n]\mr{\sigma}[n]$  extending $g$, so $\sigma (l) = g(l)$ for $l=1,
\ldots ,k$. The set of such 
$\sigma$ 's, we denote $C(g)$. The set $C(g)$ has cardinality 
$(n-k)!$, by simple combinatorics.
We  clearly have 
\footnote{Notation: $[k] \mmr{g} [n]$ denotes an injective function; 
the sum ranges over all such.}
$$
S_{n}= \coprod _{[k] \mmr{g} [n]} C(g).
$$   
Therefore,  we may rewrite $\sum _{\sigma \in S_{n}}\sigma \cdot 
\omega _{f,g}
$ as 
follows  

$$
 \sum _{[k] \mmr{g} [n]} \; \sum_{\sigma \in C(g)} \; 
                           \prod _{l\in [k]}x_{f(l)\sigma(l)} 
    = \sum _{[k] \mmr{g} [n]} \; \sum_{\sigma \in C(g)} \; 
                           \prod _{l\in [k]}x_{f(l)g(l)}
$$
since for each  $g$ and each $\sigma \in C(g)$, $\sigma (l) = g(l)$,
for 
 $l\in [k] \subseteq [n]$. Therefore, for a  given $g$, the terms in
the summation 
 over $C(g)$   are equal, and there are $(n-k)!$ of them, so the
equation continues
$$
= \sum_{[k] \mmr{g} [n]} (n-k)! \; \prod _{l\in [k]}x_{f(l)g(l)} 
=(n-k)! \sum_{[k] \mmr{g} [n]} \;  \prod _{l\in [k]}x_{f(l)g(l)}\, ,
$$
and this expression  
equals $\displaystyle{(n-k)! \prod_{l\in [k]}s_{f(l)}}$ by equation 
(\ref{twox}).
Dividing by $n!$ now gives the desired equation.
\end{proof}
From the Proposition, we may deduce (recall that ${\mathbb 
Q}$ is a subring of  
$\AQ$)
\begin{prop}\label{main2x}
 Every column symmetric admissible polynomial can be 
 expressed in $\AQ_a[M^{m \times n}]$ as a polynomial in the
$s_{i}\,$'s. (This expression can be interpreted as an expression,
modulo the ideal $I$ of inadmissibles, in the polynomial ring 
 $\AQ [M^{m \times n}]$.)
 \end{prop}
 \begin{proof} Any admissible polynomial 
  $h \in \AQ_a[M^{m \times n}]$ is a 
 linear combination of admissible monomials, and $\sy$ is linear; by 
 Proposition \ref{main1x}
  $\sy$ of an admissible monomium is a polynomial in the 
 $s_{i}$\,'s. Therefore also  $\sy (h)$ is so.
 If $h$ is furthermore column symmetric, 
 $h =\sy (h)$, then $h$ itself is expressed as a polynomial 
 of the $s_{i}$\,'s,
 \mbox{ $h = G(s_{1}, \ldots ,s_{m})$} for some polynomium 
 $G \in \AQ([m]) = \AQ[y_{1}, \ldots ,y_{m}]$.
\end{proof}
 
\vspace{1ex}
 
 We  shall formulate the results so far and some of its consequences
in the category ${\mathcal A}$ of commutative $\AQ$-algebras. 

Consider the algebra
 $\AQ[y_{1}, \ldots ,y_{m}]$. Since it is the free algebra in the
 generators $y_{i}$, and \mbox{$s_{i }\in \sy (\AQ [M^{m \times n}])$,}  there is a unique algebra map (preserving degree) 
$S: \AQ[y_{1}, \ldots  ,y_{m}] \to 
                  \sy (\AQ [M^{m \times n}]) \subseteq 
                       \AQ [M^{m \times n}])$,
namely the one which sends $y_{i}\in \AQ[y_{1}, \ldots
,y_{m}]$ to $s_{i}$. 

{Let $J$ be the ideal in $\AQ[y_{1}, \ldots , y_{m}]$ generated by the monomials of degree $n+1$.}  
The quotient algebra $\AQ[y_{1}, \ldots ,y_{m}]/J$ may be 
identified with the linear subspace (not a subalgebra)  
\mbox{$\AQ_{\leq n}[y_{1}, \ldots ,y_{m}] \subseteq \AQ [y_{1}, \ldots ,y_{m}]$} 
of polynomials of degree less or equal to
$n$ {, the algebra structure given by the multiplication table
 $ \{ (x_{f(1)}\, x_{f(2)} \, \ldots \, x_{f(n+1)} = 0 \}_{f:[n+1] \to [m] }$}, 
 and no other relations.

It follows immediately from the respective multiplication 
tables (alternatively since $S$ sends the ideal $J$ into the ideal $I$) 
that we have an algebra map:

\vspace{-3ex}

\begin{equation} \label{fivex}
 \AQ_{\leq n}[y_{1}, \ldots ,y_{m}] \mr{s}
  \sy \AQ_a [M^{m \times n}]
 \end{equation} 
 making  the diagram below commutative: 
 \begin{equation} \label{smaps}
 \xymatrix
      {
       \AQ[y_{1}, \ldots  ,y_{m}] \ar[r]^S \ar@<1ex>@{>>}[d]
      & \sy (\AQ [M^{m \times n}]) \ar@<1ex>@{>>}[d]        \\
         \AQ_{\leq n}[y_{1}, \ldots  ,y_{m}]             
                                                     \ar[r]^{s}
       & \sy (\AQ_a [M^{m \times n}])      }
\end{equation}
The vertical maps are quotient maps which discard terms of degree 
 $>n$, respectively  inadmissible terms.
 Thus the map $s$ discards the inadmissible terms from the values of $S$.

 \begin{prop}\label{injx} 
The algebra map $s$  is injective.
\end{prop}
\begin{proof}
(We refer to the last paragraph in Subsection 1.1 for the notation 
$\lVert f \rVert$ for the orbit of $f$ under precomposition with 
permutations.)
Clearly the monomials 
$\omega_{\lVert f \rVert}$ of degree $\leq n$ make up a vector basis of 
$\AQ_{\leq n}[y_{1}, \ldots  ,y_{m}]$.
We may define an equivalence relation $\sim$ on the set of monomials of degree $k$ in $A_a[M^{m \times n}]$, namely 
 $\omega _{f, g} \sim \omega _{f', g'}$ iff 
 $\lVert f \rVert = \lVert f' \rVert$. We let 
$B_{\lVert f \rVert}$ be the equivalence class defined by 
$\lVert f \rVert$.
It follows that 
 $A_{\leq n}[M^{m \times n}]$ is a direct sum of the 
 subspaces  $V_{\lVert f \rVert }$  spanned by the
 $B_{\lVert f \rVert}$.
 We show that $s(\omega_{\lVert f \rVert })$ 
 lies in $V_{\lVert f \rVert }$;  
 recall equation (\ref{twox}) and note that for any $g$,  
 $\omega_{f,g} \in B_{\lVert f \rVert }$: 
$$
s(\omega_{\lVert f \rVert })  = 
 \prod_{l\in [k]}s_{f(l)} = 
 \sum_{[k] \mmr{g} [n]}  \; \prod _{l\in [k]}x_{f(l)g(l)} = 
 \sum_{[k] \mmr{g} [n]}  \; \omega_{f, g} .
$$ 
From the direct-sum 
 property of these linear subspaces, the injectivity of $\,s\,$ follows.
\end{proof}
\begin{remark}
 {\em A similar argument proves that also the map \mbox{$S: \AQ [y_{1},\ldots , 
 y_{m}] \to  \AQ[M^{m\times n}]$} is 
 injective.}  
 \end{remark} 
 
The surjectivity of the map $\,s\,$ is a reformulation of Proposition
\ref{main2x}. Thus, combining Propositions \ref{main2x} and  \ref{injx}, we have our main 
result:
  
\begin{thm} \label{MAIN}
The algebra map $s$ in (\ref{fivex}) is an isomorphism.
\end{thm}

We shall paraphrase this in geometric terms: 

 \section{Geometric interpretation}\label{orbitx} 
  
 \subsection{The category of $\AQ$-algebras and its dual}
 
 The following Section only is a reminder, to fix notation etc. As 
 above,  
 ${\mathcal A}$ denotes the category of commutative $\AQ$-algebras 
 (here just called algebras.) 
 
 The dual category ${\mathcal A}^{op}$ is essentially the category of 
 affine schemes over $\AQ$.  The objects, viewed in this category, we here just  call {\em 
 spaces}, and the maps in it, we call {\em functions}. If 
 $A \in \mathcal{A}$, we denote $\overline{A} \in \mathcal{A}^{op}$ 
 the correponding space, and similarly for maps.
 
 A main object in ${\mathcal A}$  is the polynomial ring $\AQ [x]$ in one 
 variable; as a space  it is denoted 
 $R$,
 $$R:= \overline{\AQ [x]}.$$
 Because $\AQ [x]$ is the free algebra in one generator $x$, there 
 is, for any algebra $B$, 
 a 1-1 correspondence between the set of elements of $B$ and the set 
 of algebra maps $\AQ [x] \to B$, with dual notation, with the set of functions $\overline{B} 
 \to \overline{\AQ [x]}=R$. Thus, we have the basic fact: 
 
 \medskip
 
 \noindent{\em   
 elements of an algebra $B$ correspond to 
  $R$-valued functions on the space $\overline{B}$}.
  
  \medskip
 
Since $\AQ 
 [x_{1}, \ldots ,x_{n}]$ is a coproduct in ${\mathcal A}$ of $n$ copies of 
 $\AQ [x]$, it follows that 
 $\overline {\AQ [x_{1}, \ldots ,x_{n}]} = R^n$, the ``$n$-dimensional vector  space over $R$'', product of $n$ copies of $R$.
 Therefore, the elements of $\AQ 
 [x_{1}, \ldots ,x_{n}]$ correspond to functions $R^{n}\to 
 R$, explaining in tautological terms the relationship between 
 {\em polynomials} in $n$ variables and {\em functions} $R^{n}\to R$; all functions $R^{n}\to R$ in $\mathcal{A}^{op}$ are polynomial. 
 
 Any ideal $I$ in an an algebra $B$ gives quotient map 
 $B\twoheadrightarrow 
 B/I$, and hence in ${\mathcal A}^{op}$
defines  a monic function  
 $\xymatrix{ \overline{B/I}\ar@{^{(}->}[r] &\overline{B}}$.

 It is convenient to give names to some standard  spaces thus defined. 
 The space corresponding to $\AQ [y_{1}, \ldots ,y_{m}]/J$, where 
 $J\subset \AQ [y_{1}, \ldots ,y_{m}]$ is the ideal generated by monomials of 
 degree $n+1$,  is denoted $D_{n}(m)\subset R^{m}$,
 $$
 D_{n}(m) = \overline{\AQ_{\leq n}[y_{1}, \ldots ,y_{m}]},
$$ 
\noindent and deserves the name ``the $n$th infinitesimal neighbourhood of
\mbox{ $0 \in R^{m}$''.} In the standard description of finite limits with internal variables we have: 
 $$
 D_{n}(m) = \{(x_1,\, \ldots. \, x_m)\in R^{m} \,|\, \forall  \, 
 f: [n+1] \to [m] \;\;  x_{f(1)} \,  \, \ldots \, x_{f(n+1)} = 0 \}.
 $$
 Likewise with the ideal  
  $I\subseteq \AQ [M^{m\times n}]$ described in Section \ref{admissibles}. In this case we have
$$  
D_{1}(m)^{n} = \overline{\AQ_a [M^{m\times n}]}; 
$$
 this follows since 
 $\AQ [M^{m\times n}]/I$ is the coproduct in ${\mathcal A}$  of $n$ copies of $\AQ [x_{1}, \ldots ,x_{m}]/J$,  where $J$ now is the ideal generated by 
 monomials of degree $2$. With internal variables we have the description:
 $$
D_{1}(m)^{n} = \{(x_{1,1},\, \ldots \, x_{m,n})\in R^{m\times n} \, | \, \forall \,i \in [m],\; i' \in [m],\; j \in [n] \;\; x_{i,j}\, x_{i',j} = 0 \}
$$ 
which is easily understood by the isomorphism $R^{m \times n} = (R^m)^n$. 
 
\subsection{Orbit space}
 
 Let $B$ be an algebra, and let $S$ be a
finite group acting on $B$. The subalgebra
 $\sy_{S}(B) \subseteq B$ of {\em invariant} or {\em 
 symmetric} elements 
 %(see Section \ref{prelim}) 
 may be described in 
 categorical terms, in the category ${\mathcal A}$, as the 
 joint equalizer of the automorphisms of the form 
 $B \mr{\sigma} B$ over all the $\sigma , \sigma' \ldots \in S$, 
$$
 \xymatrix
     {
       \sy_{S}(B) 
                 \ar@<-1pt>@{^{(}->}[r]
      & B 
         \ar@<-1.5ex>[r]^{\vdots}_{\sigma} 
         \ar@<1.5ex>[r]^{\sigma '}
      & B
     }
.$$
 In the category ${\mathcal A}^{op}$, this becomes a  joint 
 coequalizer, thus the orbit object of the action of $S$,
$$
\xymatrix
    {
       \overline{B}/S
                     \ar@{{<<}-}[r]
     & \overline{B}
                     \ar@{{<}-}@<-1.5ex>[r]^{\vdots}_{\overline{\sigma}} 
                     \ar@{{<}-}@<1.5ex>[r]^{\overline{\sigma '}}
     & \overline{B}
    }
.$$
The isomorphism $s$ in the Theorem \ref{MAIN}, see diagram (\ref{smaps}), 
is displayed in the following commutative diagram: 
\begin{equation} \label{trix}
 \xymatrix@C=3.6ex 
      {
        \AQ[y_{1}, \ldots  ,y_{m}]  
                                   \ar[r]^S 
                                   \ar@<1ex>@{>>}[d]
      & \sy (\AQ[M^{m \times n}])  
                                  \ar@<-1pt>@{^{(}->}[r]
                                  \ar@{>>}[d] 
      & \AQ[M^{m \times n}]  
                            \ar@<-1.5ex>[r]^{\vdots}_{\sigma} 
                            \ar@<1.5ex>[r]^{\sigma '} 
                            \ar@{>>}[d]
      & \AQ[M^{m \times n}]  
                            \ar@{>>}[d]    
       \\
         \AQ_{\leq n}[y_{1}, \ldots  ,y_{m}]  
                                             %\ar@<1ex>@{^{(}->}[u]            
                                             \ar[r]^{s}_{\cong}
       & \sy (\AQ_a [M^{m \times n}])  
                                      \ar@<-1pt>@{^{(}->}[r]
       & \AQ_a[M^{m \times n}]  
                               \ar@<-1.5ex>[r]^{\vdots}_{\sigma} 
                               \ar@<1.5ex>[r]^{\sigma '}   
       & \AQ_a[M^{m \times n}]   
     }
\end{equation}
By a tautological rewriting, diagram (\ref{trix}) becomes
\begin{equation} \label{coeqx}
\xymatrix 
      {
        R^m   
            \ar@{<-}[r]^{\overline{S}}                
            \ar@<1ex>@{{<}-^{)}}[d]              
      & (R^m)^n/S_n   
                    \ar@{{<<}-}[r]
                    \ar@{{<}-^{)}}[d] 
      & (R^m)^n  
                %\ar@/_3.5ex/[ll]_{Sum}
                \ar@{{<}-}@<-1.5ex>[r]^{\vdots}_{\overline{\sigma}} 
                \ar@{{<}-}@<1.5ex>[r]^{\overline{\sigma '}} 
                \ar@{{<}-^{)}}[d]
      & (R^m)^n  
                \ar@{{<}-^{)}}[d]    
       \\
        D_n(m)  
               \ar@{<-}[r]^{\overline{s}}_{\cong}
       & D_1(m)^n/S_n  
                      \ar@{{<<}-}[r]   
       & D_1(m)^n  
                  %\ar@/^4ex/[ll]_{sum}
                  \ar@{{<}-}@<-1.5ex>[r]^{\vdots}_{\overline{\sigma}}
                  \ar@{{<}-}@<1.5ex>[r]^{\overline{\sigma '}}   
       & D_1(m)^n   
     }
\end{equation}
The composite map $(R^{m})^{n}\to R^{m}$ in the diagram is, in 
synthetic terms: ``take an $n$ tuple of vectors in $R^{m}$, and add 
them up''. It is symmetric in 
the $n$ arguments; and it restricts to a map    
$$\su :D_{(1)}(m)^{n} \to D_{n}(m).$$

Theorem \ref{MAIN} then can be expressed as follows:
\begin{thm} \label{GMAIN}
The addition map  $\su : D_1(m)^n \to D_{n}(m)$   is the quotient 
map of $D(m)^{n}$ under permutations of the $n$ factors, 
i.e.\  is universal among $S_{n}$-symmetric maps out of $D_1(m)^{n}$. 
\end{thm}
The special case where $m=1$ was called
``the symmetric 
functions property'' in the early days of synthetic differential 
geometry (see e.g.\ Exercise 
I.4.4 in \cite{[SDG]}); in this form, it was used  
(see e.g.\ \cite{[SDG]} 
Exercise I.8.3 and I.8.4,  or  \cite{[BD]} 
Proposition 3.4) to
establish the Formal 
Integration for vector fields: extending a vector field $D\times M 
\mr{} M$ into a ``formal flow'' $D_{\infty}\times M \mr{} M$.

\medskip

\noindent {\bf Remark.} It is not hard to prove that the 
constructions and results so far  
 can be presented in a coordinate-free way, i.e.\ referring to an 
 abstract \mbox{$m$-dimensional} vector space $V$ over $R$, rather than to 
 $R^{m}$, thus replacing e.g.\ the subspace $D_{n}(m) \subseteq 
 R^{m}$ by a subspace $D_{n}(V)$; see e.g.\ \cite{[SGM]} 1.2 for the 
 definition of this subobject.

\section{Primitives for closed differential 1-forms}\label{formsx}
The following Section is sketchy, and is included to give an 
indication of the kind of motivation from synthetic differential 
geometry that lead to the 
algebraic result stated in Theorem \ref{MAIN} or Theorem \ref{GMAIN}. 
Therefore, we do not attempt to give 
the reasoning fully explained, or
 in its full generality 
(e.g.\ replacing the space $R^{m}$ by an abstract vector space 
$V\cong R^{m}$, or even by an arbitrary manifold). 
Also, some of the 
structure involved, like the ring structure on $R$ (= the co-ring 
structure on $\AQ [x]$), we shall assume known. Details may be found 
in \cite{[SGM]}, and the references therein.

Two points $x$ and $y$ in $R^{m}$ are 
called {\em first order neighbours} if $y-x \in D_{1}(m)$. In this 
case, we write $x\sim y$. The relation $\sim$ is symmetric and 
reflexive, but not transitive.
A differential 1-form $\omega$ on $R^{m}$  may synthetically be 
described as an $R$-valued function $\omega $ defined on pairs of 1st 
order neighbour points $x,y$ in $R^{m}$,  with $\omega (x,x)=0$ for 
all $x$. It is 
{\em closed} if for any three points $x,y, 
z$ with $x\sim y$, $y\sim z$ and $x\sim z$,  we have $\omega 
(x,y)+\omega (y,z)=\omega (x,z)$. Now, in $R^{m}$, the data of a 
1-form $\omega$ may be encoded by giving a function $\Omega ( -;-): 
R^{m}\times R^{m}\to R$, linear in the argument after the semicolon, 
and such that
$$\omega (x,y)= \Omega (x;y-x)$$ for $x\sim y$.
 Closedness of $\omega$ implies that the bilinear $d\Omega 
(x;-,-): R^{m}\times R^{m}\to R$  is symmetric (see Proposition 2.2.7 
in \cite{[SGM]}). Hence, by the symmetric functions property (for the 
given $m$,  and for $n=2$), or by simple polarization, we get that  
the bilinear form $d\Omega (x;-,-)$ only depends on the sum of the 
two arguments. From this, it is easy to conclude (essentially by the 
 Taylor expansion in the proof of the quoted Proposition) that $\omega 
(x,y)+\omega (y,z)$ is independent of $y$, even without assuming that 
$x\sim z$.

If $f:R^{m} \to R$ is a 
function, we get a closed 1-form $df$ on $M$ by  $df(x,y):= 
f(y)-f(x)$. If $\omega =df$, we say that $f$ is a {\em primitive} of 
$\omega$.
 We may attempt 
to find a primitive $f$ of a 
given closed 1-form $\omega$, in a neighbourhood of the form $x_{0}+ D_{n}(m)$, 
where  $x_{0}\in R^{m}$. For a chain $x_{0}\sim x_{1}\sim \ldots 
\sim x_{n}$ (with each $x_{i}\sim x_{i+1})$, we want to define 
$f(x_{n})$ by the sum
\begin{equation}\label{attemptx}\omega (x_{0},x_{1})+\omega(x_{1},x_{2})+ \ldots + \omega 
(x_{n-1},x_{n});\end{equation} is this ``definition'' of $f(x_{n})$ independent 
of the ``interpolating points'' $x_{1}, \ldots ,x_{n-1}$ ?
 We may write 
$x_{i+1}= x_{i}+d_{i+1}$ with $d_{i}\in D(V)$ \mbox{($i=0, \ldots ,n-1$)}. In this case, the 
first question is whether the proposed value of $f(x_{0}+d_{1}+ \ldots 
+d_{n})$ is independent the individual $d_{i}$s ($i<n$) and only 
depends on 
their sum. By the symmetric functions property, this will follow if 
the sum 
is independent of the order in which we take the increments $d_{i}$.
But this independence follows because closedness of $\omega$ implies
$\omega (x,x+d)+\omega (x+d,x+d+d') = \omega (x,x+d')+ \omega (x+d', 
x+d+d')$, thus two consecutive summands in the proposed chain of 
$d_{i}$s may be interchanged; and such transpositions generate the 
whole of $S_{n}$. So Theorem \ref{GMAIN} allows 
us to define $f:x_{0}+D_{n}(m) \to R$ by the formula (\ref{attemptx}). 

It is then easy to conclude that $f(y)-f(x) = \omega (x,y)$ for any 
$y\sim x$, for any $x$ in the ``formal neighbourhood of $x_{0}$'' 
(meaning the set of points which can be reached by a chain $x_{0}\sim 
x_{1}\sim \ldots \sim x$, starting in 
$x_{0}$. So $f$ is a primitive of $\omega$ on this formal 
neighbourhood.

\bigskip

\bigskip

\small

\medskip

\noindent Eduardo J.\ Dubuc,  Dept.\ de Matematicas, Univ.\ de Buenos Aires, and IMAS, UBA-CONICET, Argentina.  \\ 
\url{edubuc@dm.uba.ar}

\medskip

\noindent Anders Kock,  Dept.\ of Mathematics, Aarhus Universitet, 
Denmark\\ 
\url{kock@math.au.dk}

\bigskip

\noindent Buenos Aires, Montreal and Aarhus

\noindent December 2018

  \end{document}